\documentclass[11pt]{article}
\usepackage{amsfonts,amssymb,amsmath,amsthm}

\usepackage{longtable}
\usepackage{bm}
\usepackage{palatino}
\usepackage{mathpazo}
\usepackage{inconsolata}
\usepackage[hidelinks]{hyperref}

\usepackage{xcolor}
\usepackage{tikz}
\usetikzlibrary{patterns}
\newcommand{\ds}{\displaystyle}
\usepackage{tabularray}
\usepackage{colortbl}
\usepackage{enumitem}
\usetikzlibrary{calc}
\begingroup
    \makeatletter
    \@for\theoremstyle:=definition,remark,plain\do{%
        \expandafter\g@addto@macro\csname th@\theoremstyle\endcsname{%
            \addtolength\thm@preskip\parskip
            }%
        }
\endgroup

\usepackage[margin=1in]{geometry}

\theoremstyle{definition}
\newtheorem{theorem}{Theorem}

\newcommand*\patchAmsMathEnvironmentForLineno[1]{%
  \expandafter\let\csname old#1\expandafter\endcsname\csname #1\endcsname
  \expandafter\let\csname oldend#1\expandafter\endcsname\csname end#1\endcsname
  \renewenvironment{#1}%
     {\linenomath\csname old#1\endcsname}%
     {\csname oldend#1\endcsname\endlinenomath}}%
\newcommand*\patchBothAmsMathEnvironmentsForLineno[1]{%
  \patchAmsMathEnvironmentForLineno{#1}%
  \patchAmsMathEnvironmentForLineno{#1*}}%
\AtBeginDocument{%
\patchBothAmsMathEnvironmentsForLineno{equation}%
\patchBothAmsMathEnvironmentsForLineno{align}%
\patchBothAmsMathEnvironmentsForLineno{flalign}%
\patchBothAmsMathEnvironmentsForLineno{alignat}%
\patchBothAmsMathEnvironmentsForLineno{gather}%
\patchBothAmsMathEnvironmentsForLineno{multline}%
}

\usepackage{parskip}
\usepackage{lineno}
\usepackage[small]{titlesec}

\usepackage[square,comma,numbers,sort&compress]{natbib}

\usepackage{color}
\definecolor{todocolor}{RGB}{205,235,139}
\definecolor{todo-idea}{RGB}{120,180,255}
\definecolor{todo-error}{RGB}{208,31,60}
\definecolor{todo-question}{RGB}{255,255,136}

\usepackage[colorinlistoftodos, color=todocolor, textsize=small]{todonotes}

\usepackage{colortbl}

\newcommand{\N}{\mathbb{N}}
\newcommand{\I}{\textbf{I}}
\newcommand{\true}{\texttt{T}}
\newcommand{\false}{\texttt{F}}

\DeclareMathOperator{\Res}{Res}

\renewenvironment{abstract}{
	\begin{list}{}%
	{\setlength{\rightmargin}{1in}%
	\setlength{\leftmargin}{1in}}%
	\item[]\ignorespaces\begin{small}}%
	{\end{small}\unskip\end{list}%
}

\newpagestyle{main}[\small]{
	\headrule
	\sethead[\usepage][][]
	{\sc }{}{\usepage}
}

\title{The enumeration of inversion sequences avoiding the patterns 201 and 210}
\author{
	Jay Pantone\\
	\small Department of Mathematical and Statistical Sciences\\
	\small Marquette University\\
	\small Milwaukee, WI, USA\\
	\small \texttt{jay.pantone@marquette.edu}
}

\date{}

\begin{document}
\maketitle

\begin{abstract}
	We derive the algebraic generating function for inversion sequences avoiding the patterns $201$ and $210$ by describing a set of succession rules, converting them to a system of generating function equations with one catalytic variable, and then solving the system with kernel method techniques.%
\end{abstract}


\section{Introduction}

An \emph{inversion sequence} of length $n$ is a sequence $e = e(1)\cdots e(n)$ of non-negative integers with the property that $0 \leq e(i) < i$ for all $i$. For example, $0002034$ is an inversion sequence of length $7$. Inversion sequences of length $n$ are in bijection with permutations of length $n$ under the map $\pi \mapsto e(1)\cdots e(n)$ where $e(i)$ is the number of entries out of $\pi(1), \ldots, \pi(i-1)$ that have value greater than $\pi(i)$; see~\cite[Section 5.1.1]{knuth:taocp-3}.

For any word $w = w(1)\cdots w(n)$ over the non-negative integers, we define the \emph{standardization} of $w$ to be the word in which the smallest entries of $w$ are replaced by $0$, the next smallest are replaced by $1$, and so on. For example, the standardization of $43471993$ is $21230441$.

We define an \emph{inversion pattern}, or just \emph{pattern}, to be a word $p = p(1)\cdots p(k)$ over the non-negative integers $\N$ that contains all values between $0$ and $\max(p)$. For example, $101$ is an inversion pattern, but $202$ is not because it does not use every value between $0$ and its maximum. Note that an inversion pattern is not necessarily an inversion sequence.

We say that an inversion sequence $e$ of length $n$ \emph{contains the inversion pattern $p$} if there exists a (not necessarily consecutive) subsequence $e(i_1)e(i_2)\cdots e(i_k)$ of entries in $e$ whose standardization is $p$. If $e$ does not contain $p$ then we say $e$ \emph{avoids} $p$. For instance, the inversion sequence $0002034$ contains the pattern $102$ because $e(4)e(5)e(7) = 204$, the standardization of which is $102$. The same inversion sequence avoids the pattern $011$ because there are no three entries $e(i), e(j), e(k)$ with $i < j< k$ and $e(i) < e(j) = e(k)$. This notion of pattern containment is often called \emph{classical}, as opposed to the \emph{consecutive} pattern containment studied by Auli and Elizalde~\cite{auli:consecutive-inversions-1, auli:consecutive-inversions-2} in which only patterns that occur in consecutive positions are considered.

Taking inspiration from the active area of pattern-avoiding permutations, the study of pattern-avoiding inversion sequences was initiated in 2015 by two independent groups: Corteel, Martinez, Savage, and Weselcouch~\cite{corteel:pattern-inversions-1} and Mansour and Shattuck~\cite{mansour:inv-seqs-orig}. For a set of inversion patterns $B$, we define $\I(B)$ to be the set of inversion sequences that avoid all of the patterns in $B$. We call sets of the form $\I(B)$ \emph{inversion classes} and we call $B$ the \emph{basis} of an inversion class, mirroring the terminology used for pattern-avoiding permutations. We further use $\I_n(B)$ to denote the set of inversion sequences of length $n$ in $\I(B)$ and define the \emph{counting sequence} of $\I(B)$ to be sequence
\[
	|\I_0(B)|, |\I_1(B)|, |\I_2(B)|, \ldots,
\]
taking the convention that $\I_0(B) = \{\varepsilon\}$ where $\varepsilon$ denotes the \emph{empty sequence} of length $0$.

The initial works \cite{corteel:pattern-inversions-1} and \cite{mansour:inv-seqs-orig} focus on inversion classes $\I(B)$ where $B$ contains a single pattern of length $3$. There are 13 such patterns: $000$, $001$, $010$, $011$, $100$, $101$, $110$, $012$, $021$, $102$, $120$, $201$, and $210$. For the similar question for pattern-avoiding permutations, there are only six cases avoiding a pattern of length three and they are all enumerated by the Catalan numbers. Here, things are more interesting. Table~\ref{table:principal-inv-classes} shows the known results for these 13 classes. Many have counting sequences that already appeared in the Online Encyclopedia of Integer Sequences (OEIS)~\cite{oeis} before these works, although several were new. For several of them, only a recurrence or functional equation in three variables is known, leaving open the question of whether their generating functions are algebraic, D-finite, etc.

\begin{table}
\begin{center}
	\begin{tblr}{
		colspec= {Q[c]Q[c]Q[c]Q[c]},
		hlines = {black, 1pt},
		rowsep = 4pt,
		stretch = 0,
		row{1} = {font=\fontsize{10pt}{12pt}\bfseries\sffamily},
	}
		\SetRow{gray!15, rowsep=8pt}
		pattern & enumeration & OEIS entry & reference(s)\\
		$000$ & Euler up-down numbers & A000111 & \cite{corteel:pattern-inversions-1}\\
		$001$ & $2^{n-1}$ & A000079 & \cite{corteel:pattern-inversions-1}\\
		$010$ & recurrence in three variables & A263779 & \cite{testart:inv-010}\\
		$011$ & Bell numbers & A000110 & \cite{corteel:pattern-inversions-1}\\
		$100$ & functional equation in three variables & A263780 & \cite{kotsireas:inv-seqs-gen-tree}\\
		$101$ / $110$ & Permutations avoiding $1-23-4$ & A113227 & \cite{corteel:pattern-inversions-1}\\
		$012$ & odd-indexed Fibonacci numbers & A001519 & \cite{corteel:pattern-inversions-1, mansour:inv-seqs-orig}\\
		$021$ & large Schr\"oder numbers & A006318 & \cite{corteel:pattern-inversions-1, mansour:inv-seqs-orig}\\
		$102$ & algebraic generating function & A200753 & \cite{mansour:inv-seqs-orig}\\
		$120$ & functional equation in three variables & A263778 & \cite{mansour:inv-seqs-orig} & \\
		$201$ / $210$ & functional equation in three variables & A263777 & \cite{corteel:pattern-inversions-1, mansour:inv-seqs-orig, kotsireas:inv-seqs-gen-tree}
	\end{tblr}
	\end{center}
	\caption{Enumerative results for the thirteen inversion classes avoiding one pattern of length three. Note that there are two pairs with the same counting sequence.}
	\label{table:principal-inv-classes}
\end{table}

Since then, substantial effort has been made to enumerate inversion classes involving more than one length three pattern. Yan and Lin~\cite{yan:inv-pairs} undertook a systematic study of the $78$ inversion classes avoiding two patterns of length three, presenting a helpful table of the previously known enumerative results (see~\cite{bouvel:inv-100-210, cao:inv-seqs, martinez:inv-seqs-2}) and deriving the enumeration of $36$ more\footnote{Callan, Jel\'inek, and Mansour~\cite{callan:inversions-avoiding-3x3} later studied inversion classes avoiding three inversion patterns of length $3$ and Mansour~\cite{mansour:five-inv-classes-1x4} has studied several inversion classes avoiding one pattern of length $4$.}. Taking some inspiration again from the area of permutation patterns, we will call these the \emph{2$\times$3 classes}. Shortly thereafter, Kotsireas, Mansour, and Y{\i}ld{\i}r{\i}m~\cite{kotsireas:inv-seqs-gen-tree} enumerated eight more, Testart~\cite{testart:inv-010} computed another five, and then Testart~\cite{testart:pp2023-talk} announced results for eighteen more at the conference Permutation Patterns 2023.

At this point it was initially believed that all 2$\times$3 classes had been enumerated, but due to a miscommunication the class $\I(201, 210)$ remained unsolved\footnote{After posting this work on the arXiv, we were informed by Zhicong Lin~\cite{lin:personal-communication} that this inversion class actually \emph{has} been enumerated in~\cite{chen:length-5-patts-and-inv-201-210} as Theorem 4.8. Our methods differ from theirs in that our enumeration is derived from a set of succession rules that use a single integer label, while their results involve succession rules with two integer labels and a two-parameter recurrence. Although we performed a literature search, checked the relevant OEIS entry, and consulted with authors of several other papers on inversion sequences, we were unable to locate this result before writing this article.}. Martinez and Savage~\cite{martinez:inv-seqs-2} proved that the counting sequence for this class is the same as that for the quadrant marked mesh patterns introduced by Kitaev and Remmel~\cite{kitaev:quadrant-marked-mesh} which are listed in the OEIS as entry A212198, but the counting sequence of quadrant marked mesh patterns had not been found, leaving the enumeration question open. In this article, we derive the following generating function for $\I(201, 210)$.

\begin{theorem}
	\label{theorem:gf}
	The generating function for $\I(201, 210)$ is 
	\[
		\frac{2-x-\sqrt{x^2-8x^3}}{2(1-2x+2x^2)} = 1 + x + 2x^2 + 6x^3 + 24x^4 + 116x^5 + 632x^6 + 3720x^7 + \cdots.
	\]
\end{theorem}

Section~\ref{section:structure} describes the structure of $\I(201, 210)$. Section~\ref{section:rules} translates this structure into a set of succession rules and Section~\ref{section:system} further translates these succession rules into a system of generating function equations which we solve in Section~\ref{section:kernel} with the kernel method. We end with some concluding remarks, including a conjectured algebraic generating function for $\I(010, 102)$ and a corrected functional equation for $\I(011, 201)$.

\section{The structure of inversions avoiding 201 and 210}
\label{section:structure}

An inversion sequence $e = e(1)\cdots e(n)$ can be depicted graphically as the set of points $(i, e(i))$. Figure~\ref{figure:inv-pic-1} shows the inversion sequence $00002204535377896966 \in \I_{20}(201, 210)$. We call any entry that is at least as big as all entries to its left a \emph{left-to-right maximum}. In the inversion sequence in Figure~\ref{figure:inv-pic-1}, the left-to-right maxima are marked with a double circle. With such pictures in mind, we use terms like ``above'' and ``below'' to mean larger or smaller entries; for example, the left-to-right maxima are the entries for which there is no entry to its left that is strictly above it.

\begin{figure}
	\begin{center}
	\begin{tikzpicture}[scale=0.5]
		\foreach \i in {0, 1, ..., 19}{
			\foreach \j in {0, 1, ..., 9}{
				\draw[gray, fill=gray] (\i,\j) circle (0.7pt);	
			}
		}
		\foreach \v [count=\i] in {0, 0, 0, 0, 2, 2, 0, 4, 5, 3, 5, 3, 7, 7, 8, 9, 6, 9, 6, 6}{
			\draw[fill] ({\i-1},\v) circle (5pt) node[below=2pt, font=\scriptsize] {\v};	
		}
		\foreach \i/\v in {0/0, 1/0, 2/0, 3/0, 4/2, 5/2, 7/4, 8/5, 10/5, 12/7, 13/7, 14/8, 15/9, 17/9}{
			\draw (\i, \v) circle (8pt);
		}
	\end{tikzpicture}
	\end{center}
	\caption{The inversion sequence $00002204535377896966$ of length $20$ that avoids the patterns $201$ and $210$. The left-to-right maxima are marked with a double circle.}
	\label{figure:inv-pic-1}
\end{figure}

\begin{theorem}
	\label{theorem:structure}
	An inversion sequence $e$ avoids the patterns $201$ and $210$ if and only if the entries strictly below and to the right of any left-to-right maximum (if there are any) all have the same value.
\end{theorem}
\begin{proof}
	We prove each direction by contrapositive. First assume $e = e(1)\cdots e(n)$ is an inversion sequence with left-to-right maximum $e(i)$ and two entries $e(j)$ and $e(k)$ that are strictly below and to the right of $e(i)$ and have different values, i.e., $i < j < k$, $e(j) < e(i)$, $e(k) < e(i)$, and $e(j) \neq e(k)$. Either $e(j) < e(k)$ in which case the entries $e(i), e(j), e(k)$ form a $201$ pattern, or $e(j) > e(k)$ in which case the entries form a $210$ pattern. In either case, $e \not\in \I(201, 210)$.
	
	To prove the other direction, assume $e \not\in \I(201, 210)$ and let $e(i), e(j), e(k)$ be three entries with $i < j < k$ that form either a $201$ or $210$ pattern. If $e(i)$ is a left-to-right maximum then $e(j)$ and $e(k)$ are two entries strictly below and to the right of the left-to-right maximum $e(i)$ and they do not have the same value. If, however, $e(i)$ is not a left-to-right maximum, then there must be a left-to-right maximum $e(i')$ with $i' < i$ and $e(i') > e(i)$. Now again, $e(j)$ and $e(k)$ are two entries with different values strictly below and to the right of $e(i')$.
\end{proof}

In this section we will describe a left-to-right construction procedure that builds each $e \in \I(201, 210)$ in a unique way. The \emph{maximum value} of any inversion sequence $e$ of length $n$ is $\max(e(1), e(2),\allowbreak  \ldots, e(n))$. For any $e \in \I(201, 210)$, we call the entries that have maximum value the \emph{big entries}. Note that big entries are left-to-right maxima, but not all left-to-right maxima are big entries. If there are any entries strictly below and to the right of any big entries, Theorem~\ref{theorem:structure} implies they all have the same value, and we call these the \emph{little entries}. We will really only be concerned with the value of the big and little entries, which we call the \emph{big value} and the \emph{little value}. In the inversion sequence in Figure~\ref{figure:inv-pic-1}, the big value is $9$ and the little value is $6$. On the other hand, the inversion sequence obtained by removing the last four entries of the inversion sequence in Figure~\ref{figure:inv-pic-1} would still have big value $9$ but would have no little value.

We can now begin to describe how the inversion sequences avoiding $201$ and $210$ can be built from left-to-right. Let $e = e(1)\cdots e(n) \in \I_n(201, 210)$ and consider the possible values $e(n+1)$ for which $e^+ = e(1) \cdots e(n)e(n+1)$ still avoids $201$ and $210$.

If $e$ has big value $b$, then one possibility is $e(n+1) = b$, adding an additional big entry of the same value. Another possibility is any value $e(n+1) \in \{b+1, b+2, \ldots, n\}$, making $e(n+1)$ a new left-to-right maximum and the new big value. If $e$ has little value $\ell$, then a third possibility is that we can add a little entry of the same value, $e(n+1) = \ell$. For example, letting $e$ be the inversion sequence in Figure~\ref{figure:inv-pic-1}, the options are $e(21) = 9$ (another big entry with the same value), $e(21) \in \{10, 11, \cdots, 20\}$ (a new big entry with a new big value up to the maximum allowed by the inversion sequence condition), and $e(21) = 6$ (a new little entry with the same little value).

If $e$ does not have a little value, as in the inversion sequence shown in Figure~\ref{figure:inv-pic-2}, then $e(n) = b$ and we can choose a new little value. Let $e(j)$ be the rightmost entry that is not a left-to-right maximum and let $e(i)$ be the rightmost left-to-right maximum to the left of $e(j)$. In Figure~\ref{figure:inv-pic-2}, $e(j) = e(9) = 3$ and $e(i) = e(8) = 6$. We now consider various cases for a new little value.
\begin{enumerate}[label=$\diamond$]
	\item If $e(n+1) < e(j)$, then the entries $e(i), e(j), e(n+1)$ would form a $210$ pattern, and so this option is not permitted.
	\item If $e(n+1) = e(j)$, Theorem~\ref{theorem:structure} implies that no $210$ or $201$ patterns are created, so this option is allowed.
	\item If $e(j) < e(n+1) < e(i)$, then the entries $e(i), e(j), e(n+1)$ would form a $201$ pattern, and so this option is not permitted.
	\item If $e(i) \leq e(n+1) < e(n)$, Theorem~\ref{theorem:structure} implies that no $210$ or $201$ patterns are created, so this option is allowed.
\end{enumerate}
For the inversion sequence in Figure~\ref{figure:inv-pic-2}, the values $e(n+1) \in \{3, 6, 7, 8\}$ are valid choices for a new little value, while the values $e(n+1) \in \{0, 1, 2, 4, 5\}$ are not.

\begin{figure}
	\begin{center}
	\begin{tikzpicture}[scale=0.5]
		\foreach \i in {0, 1, ..., 12}{
			\foreach \j in {0, 1, ..., 9}{
				\draw[gray, fill=gray] (\i,\j) circle (0.7pt);	
			}
		}
		\foreach \v [count=\i] in {0, 0, 2, 3, 1, 3, 6, 6, 3, 8, 8, 9, 9}{
			\draw[fill] ({\i-1},\v) circle (5pt) node[below=2pt, font=\scriptsize] {\v};	
		}
	\end{tikzpicture}
	\end{center}
	\caption{The inversion sequence $0023136638899$ which has big value $9$ and no little value.}
	\label{figure:inv-pic-2}
\end{figure}
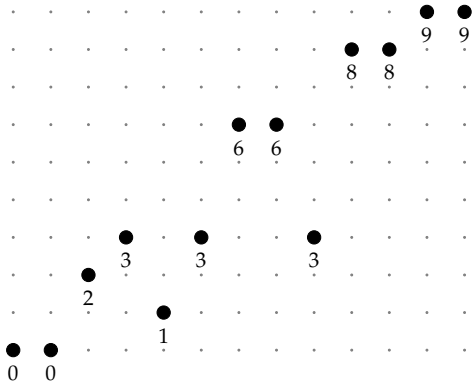

Summarizing the above discussion, when building an inversion sequence in $\I(201, 210)$ from left-to-right, at each step the options are (1) add a new big entry with the current big value, (2) add a new big entry larger than the current big value, or (3) add a little entry with either the current little value if there is one, or a little entry with a new little value if there is not.

\section{Enumeration via succession rules}
\label{section:rules}

To enumerate $\I(201, 210)$ we will derive a set of succession rules that permit, in the style of dynamic programming, the computation of $|\I_n(201, 210)|$ in quadratic time.\footnote{The next section shows how these succession rules can be converted into a generating function, which then itself leads to a linear-time algorithm.} We highly recommend Conway's survey~\cite{conway:dynamic-programming} of enumeration algorithms of this style.

Imagine constructing a $201,210$-avoiding inversion sequence from left-to-right. The succession rules will describe how new entries can be appended, as described in the previous section, but with a twist. Instead of adding a new, larger left-to-right maximum and then later possibly adding a new little value below it, we will decide at the moment of adding the new left-to-right maximum whether there will ever be a new little value below, and if so what its value will be. We call this a \emph{commitment}. For an example of what we mean by this, consider the inversion sequence $000022045353$ formed by the first $12$ entries of the inversion sequence in Figure~\ref{figure:inv-pic-1}. The allowable next steps are:
\begin{enumerate}[label=$\diamond$]
	\item Another little entry of the same value, $3$;
	\item Another big entry of the same value, $5$;
	\item A new big entry with one of the values $\{6, \ldots 12\}$, keeping the same little value $3$ (i.e., not committing to a new value);
	\item A new big entry with value $b \in \{6, \ldots, 12\}$, committing to a particular new little value $\ell \in \{5, \ldots, b-1\}$.
\end{enumerate}

The choice actually taken in the construction of the inversion sequence in Figure~\ref{figure:inv-pic-1} is a new big entry with value $7$ with a commitment to a new little value of $6$. Once we have committed to a new little value, we must at some point in the future actually place that little value, otherwise we will end up over-counting many inversion sequences. Note that a $6$ is not actually used until after new left-to-right maxima with values $8$ and $9$. This is key to our approach: if a little value $\ell$ is going to be used, we commit to it when we jump from a left-to-right maxima with value $v \leq \ell$ to a left-to-right maxima with value $v' > \ell$. 

The inversion sequences in each stage of the construction process will be represented by 3-tuples (which we call \emph{states}) of the form $s = (k, \ell, c)$, where $k$ is a non-negative integer and $\ell$ and $c$ are boolean values. The integer $k$ represents the difference between the current length of the inversion sequence and its maximum value. We call this quantity the \emph{bounce} of an inversion sequence, as it is the number of possible values that a new left-to-right maximum (larger than the previous left-to-right maxima) can take. For example, the inversion sequence $000121$ has bounce $k = 6 - 2 = 4$, representing that a new rightmost entry that is also a new maximum could take four possible values---$3$, $4$, $5$, or $6$---while still possessing the inversion sequence property that $e(i) < i$ for all $i$. The boolean $\ell$ is true if we have ever committed to a little entry at any point, and equivalently $\ell$ is false at the end of the construction if and only if the inversion sequence contains only left-to-right maxima. The boolean $c$ is true when we have committed to a particular new little value but have not placed an entry of that value yet, and then $c$ becomes false once that commitment has been fulfilled. We reiterate that if we are going to eventually use a little value of $p$, then we must commit to it the moment that we insert the first left-to-right maximum with value greater than $p$, even if we will not use the little value $p$ until after placing larger left-to-right maxima.
 
 More formally, each state represents a set of inversion sequences, some of which are augmented with an unfulfilled commitment to use a particular value. We use the notation $e/c_i$ to represent an inversion sequence $e$ paired with a commitment to use the value $i$ as a future entry. For emphasis, we write $e/-$ to represent the inversion sequence $e$ with no open commitment.
 
\begin{itemize}
	\item The state $(k, \false, \false)$ represents inversion sequences $e/-$  with bounce $k$ in which every entry is a left-to-right maximum. In other words, no little entry has ever been placed.
	\item The state $(k, \true, \false)$ represents inversion sequences $e/-$ that have bounce $k$, no open commitment, and contain at least one entry that is not a left-to-right maximum. These are the inversion sequences in which a little entry has been placed.
	\item The state $(k, \true, \true)$ represents inversion sequences $e/c_i$ that have bounce $k$ and do have an open commitment. The inversion sequence $e$ may or may not yet have an entry that is not a left-to-right maximum, but we have committed to place such an entry with value $i$ in a future step.
\end{itemize} 
  
Before listing the general form of the succession rules, we go through one example in detail, describing the sequence of states encountered during the construction of the inversion sequence $e = 0023136638899$ from Figure~\ref{figure:inv-pic-2}. The table below shows the sequence of states encountered building $e$.

\begin{longtable}{ccp{4in}}
	entry added & resulting state & notes\\\hline\endhead
	[start state] & $(0, \false, \false)$  & \\\ \\[-10pt]
	$e(1) = 0$ & $(1, \false, \false)$ & $k=1$ because if the next entry is larger there is one possible value for it ($1$); $\ell = \false$ because we have inserted no little entries; $c = \false$ because we have made no commitments.\\\ \\[-10pt]
	$e(2) = 0$ & $(2, \false, \false)$ & \\\ \\[-10pt]
	$e(3) = 2$ & $(1, \true, \true)$ & Commit to little value $1$; $\ell=\true$ now and for the rest of the construction because a little value has been committed to; $c=\true$ because we are committing to a little value, even though we will not actually use it until after another larger left-to-right maximum is inserted. Note that the state does not record which little value we have committed to.\\\ \\[-10pt]
	$e(4) = 3$ & $(1, \true, \true)$ & We still have $c=\true$ because we have not yet used the little value $1$.\\\ \\[-10pt]
	$e(5) = 1$ & $(2, \true, \false)$ & Now $c=\false$ because we have used the little value and satisfied the commitment.\\\ \\[-10pt]
	$e(6) = 3$ & $(3, \true, \false)$ & \\\ \\[-10pt]
	$e(7) = 6$ & $(1, \true, \true)$ & Commit to little value $3$.\\\ \\[-10pt]
	$e(8) = 6$ & $(2, \true, \true)$ & \\\ \\[-10pt]
	$e(9) = 3$ & $(3, \true, \false)$ & $c=\false$ because the commitment has been satisfied.\\\ \\[-10pt]
	$e(10) = 8$ & $(2, \true, \false)$ & We have added a new, larger left-to-right maximum but have not committed to a new little value; if we do insert a little value next it would have to be the previous value $3$.\\\ \\[-10pt]
	$e(11) = 8$ & $(3, \true, \false)$ & \\\ \\[-10pt]
	$e(12) = 9$ & $(3, \true, \false)$ & We have added a new, larger left-to-right maximum but have not committed to a new little value; if we do insert a little value next it would have to be the previous value $3$.\\\ \\[-10pt]
	$e(13) = 9$ & $(4, \true, \false)$ & 
\end{longtable}

A nuance that this table makes evident is that although we sometimes commit to a new little value, the new state does not record to which value we have committed. This will instead be reflected in the succession rules, which will state that, for example, if we have the inversion sequence $000$ represented by the state $(3, \false, \false)$, then the succeeding states are those in the multiset\footnote{Exponents represent repetition of an element.}
\[
	\{
		(4, \false, \false)^1,
		(3, \false, \false)^1,
		(2, \false, \false)^1,
		(1, \false, \false)^1,
		(3, \true, \true)^1,
		(2, \true, \true)^2,
		(1, \true, \true)^3
	\}.
\]
The first state represents $0000$, and the next three represent $0001$, $0002$, and $0003$ with no commitment to a little value, which then forbids ever using $0$, $1$, or $2$ as a little value in the rest of the inversion sequence. The element $(3, \true, \true)^1$ represents $0001$, but with a commitment to a little value, of which there is only one possibility ($0$). The elements $(2, \true, \true)^2$ represent $0002$ with a commitment to one of two little values ($0$ or $1$), while $(1, \true, \true)^3$ represents $0003$ with a commitment to one of three little values ($0$, $1$, or $2$).

We will now list the succession rules that generate all inversion sequences in $\I(201, 210)$. As before, an exponent means that the rule produces multiple copies of a state. In parentheses we explain what kind of new entry each rule inserts.
\begin{align*}
	(k, \false, \false) \longrightarrow \; & (k+1, \false, \false) && \text{(big entry with current big value)}	\\
		& (i, \false, \false), \quad i \in \{1,\ldots, k\} && \text{(big entry with larger value, no commitment)}\\
		& (i, \true, \true)^{k-i+1}, \quad i \in \{1, \ldots, k\} && \text{(big entry with larger value, and a commitment)}\\
	(k, \true, \false) \longrightarrow \; & (k+1, \true, \false) && \text{(big entry with current big value)}\\
		& (k+1, \true, \false) && \text{(little entry with current little value)}\\
		& (i, \true, \false), \quad i \in \{1,\ldots, k\} && \text{(big entry with larger value, no new commitment)}\\
		& (i, \true, \true)^{k-i+1}, \quad i \in \{1, \ldots, k\} && \text{(big entry with larger value, and a commitment)}\\
	(k, \true, \true) \longrightarrow \; & (k+1, \true, \true) && \text{(big entry with current big value)}\\ 
		& (k+1, \true, \false)  && \text{(little entry with committed little value)}\\
		& (i, \true, \true), \quad i \in \{1, \ldots, k\}  && \text{(big entry with larger big value, same commitment)}\\
\end{align*}

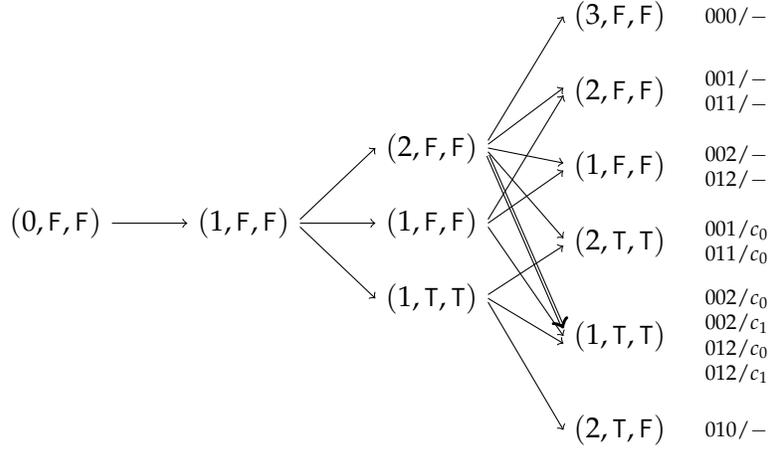
\begin{figure}
	\begin{center}
	\begin{tikzpicture}[scale=0.5]
		\node (0_0ff) at (0, 10) {$(0, \false, \false)$};
		
		\node (1_1ff) at (5, 10) {$(1, \false, \false)$};
		
		\node (2_2ff) at (10, 12) {$(2, \false, \false)$};
		\node (2_1ff) at (10, 10) {$(1, \false, \false)$};
		\node (2_1tt) at (10, 8)  {$(1, \true, \true)$};
		
		\node (3_3ff) at (15, 15.5)  {$(3, \false, \false)$};
		\node (3_2ff) at (15, 13.5)  {$(2, \false, \false)$};
		\node (3_1ff) at (15, 11.5)  {$(1, \false, \false)$};
		\node (3_2tt) at (15, 9.5)  {$(2, \true, \true)$};
		\node (3_1tt) at (15, 7)  {$(1, \true, \true)$};
		\node (3_2tf) at (15, 4.5)  {$(2, \true, \false)$};
			
		\draw[->] (0_0ff.east) -- (1_1ff.west);
		
		\draw[->] ($(1_1ff.east)+(0,0.1)$) -- (2_2ff.west);
		\draw[->] (1_1ff.east) -- (2_1ff.west);
		\draw[->] ($(1_1ff.east)+(0,-0.1)$) -- (2_1tt.west);
		
		\draw[->] ($(2_2ff.east)+(0,0.2)$) -- (3_3ff.west);
		\draw[->] ($(2_2ff.east)+(0,0.0)$) -- ($(3_1ff.west)+(0,0.1)$);
		\draw[->] ($(2_2ff.east)+(0,0.1)$) -- ($(3_2ff.west)+(0,0.1)$);
		\draw[->, double] ($(2_2ff.east)+(0,-0.2)$) -- ($(3_1tt.west)+(0,0.2)$);
		\draw[->] ($(2_2ff.east)+(0,-0.1)$) -- ($(3_2tt.west)+(0,0.1)$);
		\draw[->] ($(2_1ff.east)+(0,0.1)$) -- ($(3_2ff.west)+(0,-0.1)$);
		\draw[->] (2_1ff.east) -- ($(3_1ff.west)+(0,-0.1)$);
		\draw[->] ($(2_1ff.east)+(0,-0.1)$) -- (3_1tt.west);
		\draw[->] ($(2_1tt.east)+(0,0.1)$) -- ($(3_2tt.west)+(0,-0.1)$);
		\draw[->] ($(2_1tt.east)+(0,-0.1)$) -- (3_2tf.west);
		\draw[->] (2_1tt.east) -- ($(3_1tt.west)+(0,-0.2)$);
		
		\node[anchor=west,align=center,font=\scriptsize] at ($(3_3ff.east)+(0.5,0.0)$) {$000/-$};
		\node[anchor=west,align=center,font=\scriptsize] at ($(3_2ff.east)+(0.5,0.0)$) {$001/-$\\$011/-$};
		\node[anchor=west,align=center,font=\scriptsize] at ($(3_1ff.east)+(0.5,0.0)$) {$002/-$\\$012/-$};
		\node[anchor=west,align=center,font=\scriptsize] at ($(3_2tt.east)+(0.5,0.0)$) {$001/{c_0 }$\\$011/{c_0}$};
		\node[anchor=west,align=center,font=\scriptsize] at ($(3_1tt.east)+(0.5,0.0)$) {$002/{c_0}$\\$002/{c_1}$\\$012/{c_0}$\\$012/{c_1}$};
		\node[anchor=west,align=center,font=\scriptsize] at ($(3_2tf.east)+(0.5,0.0)$) {$010/-$};
	\end{tikzpicture}
	\end{center}
	\caption{A diagram showing all applications of at most three succession rules to the start state $(0, \false, \false)$ and the inversion sequences that they represent. The double arrow represents the fact that the state $(2, \false, \false)$ produces the next state $(1, \true, \true)$ twice.}
	\label{figure:rules}
\end{figure}

Figure~\ref{figure:rules} shows the diagram of states resulting from applying up to three succession rules in all ways to the start state $(0, \false, \false)$. The double arrow represents the succession rule $(2, \false, \false) \to (1, \true, \true)^2$ that has multiplicity two.  

For any state $(k, \ell, c)$, we denote by $P_n(k,\ell,c)$ the number of sequences of exactly $n$ succession rules, counting multiplicity, that lead from the start state $(0, \false, \false)$ to $(k, \ell, c)$. As seen in Figure~\ref{figure:rules}, $P_3(3,\false,\false) = 1$, $P_3(2,\true, \true) = 2$, and $P_3(1,\true, \true) = 4$, for example.

The inversion sequences in $\I_n(201, 210)$ are the ones represented by the states that are reached starting with the state $(0, \false, \false)$, following $n$ succession rules (counting multiplicity), and ending in a state with $c = \false$, implying that there are no unfulfilled commitments. Therefore we have
\[
	|\I_n(201, 210)| = \sum_{\substack{k \in \{0, \ldots, n\}\\\ell \in \{\true, \false\}}} P_n(k,\ell,\false).
\]
For instance, inspecting the rightmost states that have $c=\false$ shows that
\[
	|\I_3(201, 210)| = P_3(3,\false,\false) + P_3(2,\false,\false) + P_3(1,\false,\false) + P_3(2,\true,\false) = 1 + 2 + 2 + 1= 6.
\]

\section{Converting the succession rules into a system of equations}
\label{section:system}

Define
\[
	a_{n,k} = P_n(k, \false, \false), \qquad
	b_{n,k} = P_n(k, \true, \false), \qquad
	c_{n,k} = P_n(k, \true, \true),
\]
and the corresponding generating functions
\[
	A(x,u) = \sum_{n,k} a_{n,k}x^nu^k, \qquad
	B(x,u) = \sum_{n,k} b_{n,k}x^nu^k, \qquad
	C(x,u) = \sum_{n,k} c_{n,k}x^nu^k.
\]

From the succession rules in the previous section we can write a system of equations relating these three generating functions. We first remark that $A(x,1)$ is the generating function for $\I(10)$, inversion sequences in which every entry is a left-to-right maximum, and $A(x,1) + B(x,1)$ is the generating function for $\I(201, 210)$.

\begin{theorem}
	The generating functions $A(x,u)$, $B(x,u)$, and $C(x,u)$ defined above satisfy the system of equations
	\begin{align*}
		A &= 1 + xu(A + \Phi(A))\\
		B &= xu(2B + \Phi(B) + C)\\
		C &= xu(\Phi(u\Phi(A+B)) + \Phi(C) + C)	
	\end{align*}
	where we have used $A$, $B$, and $C$ as short-hand for $A(x,u)$, $B(x,u)$, and $C(x,u)$, and where $\Phi$ is defined by $\Phi(f(x,u)) = \ds\frac{f(x,1) - f(x,u)}{1-u}$.
\end{theorem}
\begin{proof}
	The start state $(0, \false, \false)$ contributes the constant term $1$ to $A(x,u)$ and no constant term to $B(x,u)$ or $C(x,u)$. All other terms on the right-hand sides come from examining the succession rules from the previous section.
	
	We start with the equation $A(x,u)$, which counts sequences of succession rules ending in a state of the form $(k, \false, \false)$. The only two succession rules whose right-hand side is a state of that form are
	\[
		(k,\false,\false) \to (k+1,\false,\false) \qquad \text{ and } \qquad (k,\false,\false) \to (i,\false,\false), \; i \in \{1,\cdots,k\}.
	\]
	The products of the first rule contribute the term $xuA(x,u)$ because applying this rule to each state represented by $A(x,u)$ produces a new state with $k$ incremented by one. The factor of $x$ represents the fact that a state on the left that was produced by $n$ applications of succession rules has now produced a state that is the result of $n+1$ applications. The factor of $u$ arises because the value of $k$ is increased by $1$.
	
	The second term $xu\Phi(A(x,u))$ requires some explanation of the function $\Phi$. Noting that
	\[
		\Phi(u^k) = \frac{1-u^k}{1-u} = 1 + u + u^2 + \cdots + u^{k-1},
	\]
	we observe that the effect of $\Phi$ on $f(x,u)$ is to replace every $u^k$ term by $1 + \cdots + u^{k-1}	$. For example
	\[
		\Phi((1+u+3u^2)x + (2u + 4u^2 + u^3)x^2) = (4 + 3u)x + (7 + 5u + u^2)x^2.
	\]
	The operator $\Phi$ is often called the \emph{discrete derivative} with respect to $u$ because
	\[
		\left(\frac{d}{du}f(x,u)\right)_{u=1} = \lim_{u \to 1} \frac{f(x,1) - f(x,u)}{1-u}.
	\]
	The succession rule $(k,\false,\false) \to (i,\false,\false),\; i \in \{1, \cdots, k\}$ corresponds to the termwise operation
	\[
		a_{n,k}x^nu^k \to a_{n,k}x^{n+1}(u + u^2 + \cdots + u^k) = xu( a_{n,k}x^n(1 + u + \cdots + u^{k-1} )) = xu (a_{n,k}x^n\Phi(u^k))
	\]
	and so the states generated by the rules are enumerated by the term $xu\Phi(A(x,u))$. Therefore
	\[
		A(x,u) = 1 + xu(A(x,u) + \Phi(A(x,u)).
	\]
	
	The second equation is derived in a similar way. There are three succession rules\footnote{According to the succession rules listed in the previous section there are four, but here we have combined two identical rules into one with multiplicity two.} that have states corresponding to $B(x,u)$ on the right-hand side:
	\begin{align*}
		(k, \true, \false) \to \; & (k+1, \true, \false)^2;\\
		(k, \true, \false) \to \; & (i, \true, \false), \quad i \in \{1,\ldots, k\};\\
		(k, \true, \true) \to \; & (k+1, \true, \false).
	\end{align*}
	The first two have states on the left-hand side that are counted by $B(x,u)$, while the state on the left-hand side of the third is counted by $C(x,u)$. For reasons similar to the explanation for $A(x,u)$, the first rule contributes $2xuB(x,u)$ to the right-hand side of the equation for $B(x,u)$, the second rule contributes $xu\Phi(B(x,u))$, and the third rule contributes $xuC(x,u)$. Hence
	\[
		B(x,u) = xu(2B(x,u) + \Phi(B(x,u)) + C(x,u)).
	\]	
	
	Finally, there are four rules that have states corresponding to $C(x,u)$ on the right-hand side:
	\begin{align*}
		(k, \false, \false) \to \; & (i, \true, \true)^{k-i+1}, \quad i \in \{1, \ldots, k\};\\
		(k, \true, \false) \to \; & (i, \true, \true)^{k-i+1}, \quad i \in \{1, \ldots, k\};\\
		(k, \true, \true) \to \; & (k+1, \true, \true);\\
		(k, \true, \true) \to \; & (i, \true, \true), \quad i \in \{1, \ldots, k\}.
	\end{align*}
	The first rule produces, from any state $(k, \false, \false)$ counted by $A(x,u)$, $k-i+1$ copies of the state $(i, \true, \true)$. This requires an operation on $A(x,u)$ that achieves the transformation
	\[
		u^k \to x\left(ku + (k-1)u^2 + (k-2)u^3 + \cdots + 2u^{k-1} + u^k\right).
	\]
	To this end, observe that
	\begin{align*}
		xu\Phi(u\Phi(u^k)) &= xu\Phi(u(1 + \cdots + u^{k-1}))\\
		&= xu\Phi(u + \cdots + u^k)\\
		&= xu\left(\Phi(u) + \cdots + \Phi(u^k)\right)\\
		&= xu\left((1) + (1 + u) + \cdots + (1 + \cdots + u^{k-1})\right)\\
		&= x\left(ku + (k-1)u^2 + \cdots + u^{k}\right)
	\end{align*}
	and so $xu\Phi(u\Phi(A(x,u)))$ counts the states produced by the first rule. For the same reason, $xu\Phi(u\Phi(B(x,u)))$ counts states produced by the second rule. The third and fourth rules are accounted for by the terms $xuC(x,u)$ and $xu\Phi(C(x,u))$, giving the equation
	\[
		C(x,u) = xu(\Phi(u\Phi(A(x,u)+B(x,u))) + \Phi(C(x,u)) + C(x,u)). \qedhere
	\]
\end{proof}

\section{Solving the system of equations}
\label{section:kernel}

The extra variable $u$ is often called a \emph{catalytic variable} because its presence is needed in order to write down this system of equations, but after solving the system we discard it by setting $u=1$ to find the univariate generating function for $\I(201, 210)$. The kernel method is a technique first used by Knuth~\cite{knuth:taocp-1} in his solution to Exercise 4 of Section 2.2.1 that solves some simple types of equations for generating functions with one catalytic variable. Bousquet-M\'elou and Jehanne~\cite{bousquet-melou:poly-eqs} call this ``the unnoticed birth of the kernel method'' and significantly generalize the method so that it can be used to solve nearly any single equation with one catalytic variable. We will now solve our system of equations by iteratively isolating the variables $A$, $B$, and $C$ into single equations and then applying this method. A Maple worksheet containing all of the calculations below and a pdf export of that worksheet are both available as ancillary files on the arXiv version of this article.

One problem we encounter is that, symbolically, $\Phi(u\Phi(F(x,u))$ leads to a $0/0$ term. Instead, we add a fourth equation
\[
	D(x,u) = \Phi(A(x,u) + B(x,u))
\]
and rewrite the third equation as
\[
	C(x,u) = xu(\Phi(uD(x,u)) + \Phi(C(x,u)) + C(x,u)).
\]

For the remainder of this section, we use shorthand $A_u$ for $A(x,u)$ and $A_1$ for $A(x,1)$, and the same for $B$, $C$, and $D$. Expanding each application of $\Phi$ and rewriting the equations in this way gives the system
\begin{align*}
	A_u &= 1 + xu\left(A_u + \frac{A_1 - A_u}{1-u}\right)\\
	B_u &= xu\left(2B_u + \frac{B_1 - B_u}{1-u} + C_u\right)\\
	C_u &= xu\left(\frac{D_1 - uD_u}{1-u} + \frac{C_1 - C_u}{1-u} + C_u\right)	\\
	D_u &= \frac{A_1+B_1-A_u-B_u}{1-u}.
\end{align*}

Now we clear denominators and rearrange the equations so they each have $0$ on the left-hand side:
\begin{align*}
	0 &= (1-u+xu^2)A_u - xuA_1 + u - 1\\
	0 &= (1-u-xu+2xu^2)B_u - xuB_1 - xu(u-1)C_u\\
	0 &= (1-u+xu^2)C_u - xuC_1 + xu^2D_u - xuD_1\\
	0 &= (1-u)D_u + A_u + B_u - A_1 - B_1.
\end{align*}

We define the polynomials $P_1$, $P_2$, $P_3$, and $P_4$ to be the four right-hand sides above. The first equation already involves only one of the four original left-hand sides, and is linear in $A_u$, allowing us to apply the simplest version of the kernel method. The coefficient $1-u+xu^2$ of $A_u$ is called the \emph{kernel} of the equation. If we substitute into $u$ a Puiseux series in $x$ that makes the kernel zero, then we obtain an equation involving only $x$ and $A_1$, which can then be substituted back into the equation to solve for $A_u$ in terms of $x$ and $u$.

Instead of explicitly solving the kernel for $u$, we can accomplish the same feat by computing the resultant of $P_1$ and $1-u+xu^2$ with respect to $u$, which has the same effect of eliminating $u$ from the equation. We find
\[
	\Res_u(P_1, 1-u+xu^2) = x^2(xA_1^2 - A_1 + 1)
\]
from which it follows that $xA_1^2 - A_1 + 1$ is a minimal polynomial for $A(x,1)$. Define $P_5 = xA_1^2 - A_1 + 1$. 

Next we use Gr\"obner basis computations\footnote{We use the \texttt{PolynomialIdeals} package in Maple.} to use $\{P_1, P_2, P_3, P_4, P_5\}$ to find an equation involving only the variables $x, u, C_u, B_1, C_1, D_1$. We find that
\[
	0 = (1 - u + 2xu^2)q_1(x,u)C_u^2 + (1 - u + 2xu^2)q_2(x,u)C_u + q_3(x,u,B_1,C_1,D_1)
\]
where $q_1$, $q_2$, and $q_3$ are polynomials. Although this equation is not linear in $C_u$, there is a factor of $1-u+2xu^2$ attached to both the $C_u$ and $C_u^2$ terms, allowing us to take the resultant of this equation and this factor with respect to $u$ to compute a polynomial $P_6(x, B_1, C_1, D_1)$ that equals $0$.

With this new equation in hand, we perform another Gr\"obner basis computation to derive a rather large polynomial\footnote{The polynomial has degree $8$ in each of $C_u$, $C_1$, and $D_1$, degree $22$ in $x$, degree $36$ in $u$, and has $29,956$ terms.} involving the variables $x, u, C_u, C_1, D_1$, into which we can directly substitute $u=1$ yielding a polynomial that factors as
\[
	x^6T_1(x,C_1,D_1)T_2(x,C_1,D_1).
\]
By computing initial terms in the power series expansions of $C_1$, and $D_1$---for example by iterating the succession rules---we can check that substituting these initial terms into $T_2$ does not give $0$, and therefore $T_1$ must be the correct factor, and so we define $P_7 = T_1$.

With $P_7$ in hand we can do one final elimination computation to produce a polynomial in terms of just $x$ and $B_1$ that factors into four terms, each term having degree $4$ in $B_1$. By checking the initial terms of the Puiseux series expansion of each factor, we find the correct factor and conclude that the minimal polynomial for $B_1$ is
\begin{align*}
	0 &= x^{2} \left(2 x^{2}-2 x +1\right)^{2} B_1^{4}+2 x \left(x -1\right) \left(3 x -1\right) \left(2 x^{2}-2 x +1\right) B_1^{3}\\
	& \qquad +\left(12 x^{5}-5 x^{4}-16 x^{3}+21 x^{2}-8 x +1\right) B_1^{2}+x \left(x -1\right) \left(6 x -1\right) \left(3 x -1\right) B_1 +x^{4}
\end{align*}

Recall that the generating function for $\I(201, 210)$ is $F(x) = A(x,1) + B(x,1)$. From the minimal polynomials for $A_1$ and $B_1$ derived above, we can deduce that the minimal polynomial of $F(x)$ is 
\[
	0 = (2x^2 - 2x + 1)F(x)^2 + (x - 2)F(x) + x + 1
\]
and therefore
\[
	F(x) = \frac{2-x-x\sqrt{1-8 x}}{2 \left(1 - 2x + 2 x^{2}\right)}
\]
thus proving Theorem~\ref{theorem:gf} and completing the enumeration of $\I(201, 210)$ and of quadrant marked mesh patterns.

\section{Concluding Remarks}

\begin{enumerate}
	\item Kotsireas, Mansour, and Y\i{}ld\i{}r\i{}m~\cite{kotsireas:inv-seqs-gen-tree} recently studied generating trees for inversion classes and they give a condition for testing whether two vertices in a generating tree are isomorphic. This allows them to algorithmically compute the generating function for inversion classes whose generating trees are finitely labeled in much the same way as Vatter's FINLABEL program~\cite{vatter:finitely-labeled} for permutation classes. Because few inversion classes satisfy this strong property, they go on to describe how their software can also be used to make a conjecture about the possible form of a infinitely labeled generating tree; sometimes the conjecture can then be rigorously proved by hand. Our approach for building inversion sequences left-to-right differs from theirs in that we are not just placing new entries, but also making and fulfilling commitments. It is this technique that allows us to describe $\I(201, 210)$ using succession rules  with a single integer component ($k$) instead of a more direct approach by hand that would lead to two integer components. The software described in~\cite{kotsireas:inv-seqs-gen-tree} is not publicly available, as far as we know, so we are unable to determine if it is able to conjecture a set of succession rules with two integer components for this class.

	\item We have updated the OEIS page for this sequence (A212198) with the generating function derived here and with initial terms of the counting sequence. We strongly encourage others who find generating functions or polynomial-time algorithms for pattern-avoiding inversion sequences to also update the corresponding OEIS sequences. Although it can be time-consuming if many sequences are involved, we feel it is a valuable resource for those who may be approaching this topic for the first time. It also would reduce the likelihood of further miscommunications about which inversion classes have been enumerated and which have not. As a step in this direction, we have used the formulas derived by Mansour and Shattuck~\cite{mansour:inv-seqs-orig} and Kotsireas, Mansour, and Y\i{}ld\i{}r\i{}m~\cite{kotsireas:inv-seqs-gen-tree} to derive 200 terms in the counting sequences for both $\I(120)$ and $\I(100)$, and added them to OEIS entries A263778 and A263780.
	
	\item In future work~\cite{pantone:inv-1x3-succession-rules} we will present succession rules for the thirteen inversion classes defined by avoiding one pattern of length three. Although at least a polynomial-time counting algorithm is now known for all thirteen, our succession rules will allow in most cases for more efficient computation of these terms. The improvements are made using a variety of methods, including the idea of ``committing'' to future values, tracking the bounce statistic using succession rules, and in some cases writing down summations instead of generating functions.

	\item At least a polynomial-time counting algorithm has also now been found (or announced) for all $78$ 2$\times$3 inversion classes. We suspect there are some of these classes for which the generating function has a simpler form than is currently known. For example, Testart~\cite{testart:pp2023-talk} gives a two-variable recurrence for the counting sequence of $\I(010, 102)$, but we conjecture that this counting sequence actually has an algebraic generating function with minimal polynomial
	\begin{align*}
		&x (x^{2}-x +1) (x -1)^{2} G(x)^{3}+2 x (x -1) (2 x^{2}-2 x +1) G(x)^{2}\\
		& \qquad -(x^{4}-8 x^{3}+11 x^{2}-6 x +1) G(x) -(2 x -1) (x -1)^{2}	= 0.
	\end{align*}
	We are able to find a pair of functional equations
	\begin{align*}
	    F(x,u) &= \frac{1}{1-xu} + \frac{xu}{1-xu}[u^{\geq 0}]\left[\frac{1}{u}F(x,u)A\left(x, \ds\frac{1}{u}\right)\right]\\
	    G(x) &= \frac{1}{1-x} + x\left([u^{\geq 0}]\left[F(x,u)\left(A\left(x, \ds\frac{1}{u}\right)-1\right)\right]\right)_{u=1}
	\end{align*}
	where $A(x,u)$ is an explicit quadratic function, $G(x)$ is the generating function for $\I(010, 102)$, and $[u^{\geq 0}]$ is the operator that extracts that non-negative powers of $u$. Unfortunately we are unable to solve the functional equations to prove our conjecture.
	
	\item In the course of writing this article, we noticed that the enumeration of $\I(011, 201)$ given by Kotsireas, Mansour, and Y{\i}ld{\i}r{\i}m~\cite{kotsireas:inv-seqs-gen-tree} is incorrect. They provide a functional equation and some initial terms that they compute from it, stating that there are 52 inversion sequences of length 5 in this class. Brute force calculation shows there are actually 51 inversion sequences of length 5. One of the authors has informed us that there is a typographical error in their functional equation~\cite{yildirim:personal-communication}. We therefore take the opportunity here to give a brief explanation of the structure of this inversion class and a correct set of succession rules.
	
	Inversion sequences that avoid the patterns $011$ and $201$ have the following two properties: (1) each nonzero value can be used at most once; (2) the entries strictly below and to the right of any left-to-right maximum must be decreasing. Thus we can build left-to-right, keeping track of two statistics, the \emph{bounce} (as defined earlier in this article), which we denote by $k$, and the number of values less than the maximum that have not yet been used, and which would not create a $201$ pattern if used, which we denote by $\ell$. The inversion sequences in $\I(011,201)$ are generated by the simple succession rules
	\begin{align*}
		(k, \ell) \longrightarrow \; & (k+1, 0)\\
		& (i, \ell+k-i), \quad i \in \{1,\ldots, k\}\\
		& (k+1, i), \quad i \in \{0,\ldots, \ell-1\}.
	\end{align*}
	As inversion sequences are built left-to-right, the first rule represents adding new entry with value $0$ which increases $k$ by $1$ but forces $\ell$ to become $0$ because no new entries can be inserted to the right that are greater than $0$ but less than the current maximum, as that would make a $201$ pattern. 
	
	The second rule represents adding a new left-to-right maximum, which decreases the bounce by some amount while also increasing $\ell$ by the number of values jumped over, as they may be used for future entries.
	
	The third rule represents placing a nonzero entry that is not a left-to-right maximum, i.e., it is smaller than the entry to its left. If $\ell=5$, for example, there are five available values for that entry. If we place it in the highest available value, there are four remaining values afterward. If we place it in the second highest, then only three remain because the value above it is no longer valid as entries placed into it would then play the role of a $1$ in a $201$ pattern.
	
	Let $F(x,u,v)$ be the generating function such that the coefficient of $x^nu^kv^\ell$ is the number of $011,201$-avoiding inversion sequences with length $n$ represented by the state $(k, \ell)$ as above. The succession rules can be converted into the functional equation
	\[
		F(x,u,v) = 1 + xu\left(F(x,u,1) + \frac{F(x,u,v) - F(x,u,1)}{v-1} + \frac{F(x,u,v) - F(x,v,v)}{u-v}\right).
	\]

	Using these succession rules we have calculated the first 500 terms in the counting sequence. They match the 27 given terms of the OEIS sequence A279555, which is defined to be the counting sequence of $\I(010, 110, 120, 210)$. Martinez and Savage~\cite{martinez:inv-seqs-2} prove that this class is equinumerous to $\I(010, 100, 120, 210)$, and Yan and Lin~\cite{yan:inv-pairs} conjecture that both are equinumerous to the class $\I(011, 201)$ for which we have just presented succession rules.\footnote{This would be an example of what Burstein and Pantone~\cite{burstein:unbalanced} call an \emph{unbalanced Wilf-equivalence} in the realm of permutation patterns, as the two inversion classes are defined by avoiding different numbers of patterns.} As far we know, this conjecture remains open. We find further evidence for the conjecture by using the succession rules below to generate 500 terms for $\I(010, 100, 120, 210)$ and confirming that they match the counting sequence for $\I(011, 201)$:
	\begin{align*}
		(k, \ell) \longrightarrow \; & (k+1, \ell)\\
		& (k+1, i), \quad i \in \{0,\ldots, \ell-1\}\\
		& (i, k-i), \quad i \in \{1,\ldots, k\}.
	\end{align*}
	We omit the details of these succession rules, but their derivation is similar to those for $\I(011, 201)$. They lead to a similar function equation
	\[
		G(x,u,v) = 1 + xu\left(G(x,u,v) + \frac{G(x,u,v) - G(x,u,1)}{v-1} + \frac{G(x,u,1) - G(x,v,1)}{u-v}\right).
	\]
	Note the full three-variable solutions $F(x,u,v)$ and $G(x,u,v)$ to the above two functional equations are not equal, but the conjecture implies that $F(x,1,1) = G(x,1,1)$. We do not know if there is a simple way to show this directly from the two functional equations.
\end{enumerate}

\section*{Acknowledgments}
We would like to thank Benjamin Testart for many helpful conversations on the topic of inversion sequences and for his feedback on this work.

This work was supported by grant 713579 from the Simons Foundation.

\bibliographystyle{alpha}
\bibliography{paper.bib}

\end{document}